\documentclass[a4paper,12pt]{amsart}
\usepackage[left=3cm,right=3cm]{geometry}

\usepackage[english]{babel}
\usepackage[utf8]{inputenc}

\usepackage{amsmath,amsfonts,amssymb,amscd,amsthm}
\usepackage[linktocpage,unicode]{hyperref}

\usepackage{latexsym,graphics,esdiff}
\usepackage[dvips]{color}
\usepackage[dvips]{graphicx}
\usepackage{pstricks,pst-tree,rotating}

\title[Flag statistics from the Ehrhart series multi-hypersimplices]{Flag statistics from the Ehrhart series of multi-hypersimplices}
\date{}
\author{Guo-Niu Han}
\address{Institut de Recherche Math\'ematique Avanc\'ee\\
Universit\'e de Strasbourg and CNRS\\
7~rue Ren\'e Descartes\\
67084 Strasbourg\\
France}
\email{guoniu.han@unistra.fr}

\author{Matthieu Josuat-Vergès}
\thanks{The second author is supported by ANR CARMA (ANR-12-BS01-0017)}
\address{Institut Gaspard Monge\\
Université Paris-Est Marne-la-Vallée and CNRS \\ 
5~Boulevard Descartes, 77454 Marne-la-Vallée CEDEX \\
France}
\email{matthieu.josuat-verges@univ-mlv.fr}

\hypersetup{
    unicode=false,          
    pdftoolbar=true,        
    pdfmenubar=true,        
    pdffitwindow=false,     
    pdfstartview={FitH},    
    pdftitle={Flag statistics from the Ehrhart series of generalized hypersimplices},
    pdfauthor={},
    pdfsubject={},
    pdfkeywords={},
    pdfnewwindow=true,      
    colorlinks=true,       
    linkcolor=red,          
    citecolor=blue,        
    filecolor=magenta,      
    urlcolor=cyan           
}

\newtheorem{theo}{Theorem}[section]

\newtheorem{lem}[theo]{Lemma}
\newtheorem{prop}[theo]{Proposition}

\theoremstyle{definition}
\newtheorem{defi}[theo]{Definition}

\DeclareMathOperator{\des}{des}
\DeclareMathOperator{\fdes}{fdes}

\DeclareMathOperator{\cdes}{cdes}

\DeclareMathOperator{\fexc}{fexc}
\DeclareMathOperator{\std}{std}
\DeclareMathOperator{\cstd}{cstd}

\begin{document}

\begin{abstract}
It is known that the normalized volume of standard hypersimplices (defined as some slices of the unit hypercube) 
are the Eulerian numbers. More generally, a recent conjecture of Stanley relates the Ehrhart 
series of hypersimplices with descents and excedences in permutations.
This conjecture was proved by Nan Li, who also gave a generalization to colored permutations.
In this article, we give another generalization to colored permutations, using the flag statistics
introduced by Foata and Han. 
We obtain in particular a new proof of Stanley's conjecture, and some combinatorial identities 
relating pairs of Eulerian statistics on colored permutations.
\end{abstract}

\maketitle


\section{Introduction}

A modern combinatorial definition of the Eulerian numbers $A_{n,k}$ is given by counting descents in permutations:
\begin{equation} \label{defank}
  A_{n,k} := \# \{  \sigma \in \mathfrak{S}_n \, : \, {\rm des}(\sigma)=k-1  \}.
\end{equation}
Foata suggested in \cite{foata} the problem that we describe below.
It is known that the Eulerian numbers $A_{n,k}$ satisfy
\[
  \frac{A_{n,k}}{n!} = {\rm Vol}(  \big\{ v\in [0,1]^n \, : \,  k-1 \leq \sum v_i \leq k  \big\} ),
\]
this is essentially a calculation due to Laplace (see \cite{foata} for details).
But the combinatorial definition can be easily translated in the following way:
\[
  \frac{A_{n,k}}{n!} = {\rm Vol}(   \{ v\in [0,1]^n \, : \,  {\rm des}(v) = k-1  \} ).
\]
The problem is to find a measure-preserving bijection between the two sets, to explain why they have the same volume. 
A simple solution was given by Stanley~\cite{stanley}.

%

The set $\{ v \in [0,1]^n \, : \, k\leq \sum v_i \leq k+1 \}$ is in fact a convex integral polytope known as 
the hypersimplex, and in this context we can consider the Ehrhart series, which is a generalization of the volume.
This led to a recent conjecture by Stanley about the Ehrhart series of 
the hypersimplex (more precisely, a partially open version of the hypersimplex), which was
proved by Nan Li \cite{li} in two different ways.
It is remarkable that two Eulerian statistics are needed to state the conjecture, which says
that the Ehrhart series of a hypersimplex is the descent generating function for permutations
with a given number of excedences.
Nan Li also extended the result to colored permutations by considering the hypercube $[0,r]^n$
for some integer $r>0$, and the polytopes $\{ v\in [0,r]^n \,:\, k\leq \sum v_i \leq k+1 \}$
are the multi-hypersimplices referred to in the title of this article.
We would like to mention that besides Stanley's conjecture, some recent works deals with the
geometry and combinatorics of hypersimplices, see \cite{hibi,ohsugi}.

The goal of this article is to give another generalization of Stanley's conjecture to colored 
permutations. Our result is stated in terms of the flag descents and flag excedences in colored permutations,
and relies on some related work by Foata and Han \cite{foatahan}. 
Our method gives in particular a new proof of Stanley's conjecture in the uncolored case.
Our method can roughly be described as follows. We first consider the case of the half-open hypercube $[0,r)^n$, 
where an analog of Stanley's conjecture in terms of descents and inverse descents can be proved in a rather elementary 
way. We can relate the half-open hypercube $[0,r)^n$ with the usual hypercube $[0,r]^n$ via an inclusion-exclusion 
argument. Then, it remains only to prove an identity relating two generating functions for colored permutations.

This article is organized as follows.
Section~\ref{sec:prelim} contains some preliminaries.
Our main results are Theorems~\ref{theobnr} and \ref{theobnr2}, whose particular case $r=1$ gives Stanley's conjecture.
The core of the proof is in Sections~\ref{sec:stanbij} and \ref{sec:inclexcl}, but
it also relies on some combinatorial results on colored permutations which are in Sections~\ref{sec:chro},
\ref{sec:genfunc} and~\ref{sec:bij}.

\section*{Acknowledgement}

We thank Dominique Foata for his corrections and comments on this article.

\section{Triangulations of the unit hypercube}

\label{sec:prelim}

This section contains nothing particularly new, but we introduce some notation and background (see \cite{beck,stanley3}).
Let $ \mathcal{X} \subset \mathbb{R}^n$ be a convex polytope with integral vertices.
The {\it Ehrhart polynomial} $E(\mathcal{X},t)$ is defined as the unique polynomial in $t$ such that,
for any integer $t>0$,
\begin{equation} \label{defehr}
  E( \mathcal{X} , t ) = \# \big(  t \mathcal{X} \cap \mathbb{Z}^n \big)
\end{equation}
where $t\mathcal{X} := \{ tx \, : \, x\in\mathcal{X} \} $.
The {\it Ehrhart series} of $\mathcal{X}$ is defined as
\begin{equation} \label{defehr2}
  E^* ( \mathcal{X}, z ) := (1-z)^{n+1} \sum_{t \geq 0} E( \mathcal{X}, t ) z^t.
\end{equation}
From a general result of Stanley \cite{stanley3}, the series $E^*(\mathcal{X},z)$ is in fact a polynomial with positive 
integral coefficients. As it often happens, it is an interesting problem to find their combinatorial meaning for special 
polytopes. Perhaps the most basic example is the unit hypercube $[0,1]^n$ which has the $n$th Eulerian polynomial as 
Ehrhart series, as will be detailed below. Although we do not use this language here, a general method to find the 
Ehrhart series 
of a polytope is to use unimodular shellable triangulations (as was done for example in \cite{li}) 
and it is essentially the idea behind what follows.

\begin{defi}
If $v=(v_1,\dots,v_n)\in\mathbb{R}^n$, let 
\[\des(v) = \#\{ i \, : 1\leq i \leq n-1, \; v_i > v_{i+1} \}.\]
We define the {\it standardization} $\std(v)$ of $v$ to be the unique permutation
$\sigma\in\mathfrak{S}_n$ such that for all $i<j$ we have
$v_i \leq v_j$ iff $\sigma_i<\sigma_j$.
For each permutation $\sigma\in\mathfrak{S}_n$, let
\[
  \mathcal{S}_{\sigma} = \{ v \in [0,1]^n \, : \, \std(v)=\sigma \}.
\]
\end{defi}

For example, one can check that if $x<y<z$, $\std(x,y,x,z,y,x,x)=1527634$.
Note that the unit hypercube $[0,1]^n$ is the disjoint union of the subsets $\mathcal{S}_{\sigma}$
for $\sigma \in \mathfrak{S}_n $.
Geometrically, each $\mathcal{S}_\sigma$ is a unit simplex where some facets are removed.
So it is not a polytope in the usual sense; we should call it a ``partially open'' polytope.
But note that Equations~\eqref{defehr} and~\eqref{defehr2} make sense even when $\mathcal{X}$ 
is not a polytope, so in particular $E^*(\mathcal{S}_\sigma,z)$
is well defined, and we have:

\begin{lem} \label{ehrssigma}
$E^*(\mathcal{S}_\sigma,z) = z^{\des(\sigma^{-1})}$.
\end{lem}

\begin{proof}
We can check that $v \in [0,1]^n$ is in $\mathcal{S}_{\sigma}$
if and only if $v_{\sigma^{-1}(1)} \leq \dots \leq v_{\sigma^{-1}(n) }  $,
and $v_{\sigma^{-1}(i)} < v_{\sigma^{-1}(i+1) }$ if $\sigma^{-1}(i)>\sigma^{-1}(i+1)$.
The number $E(\mathcal{S}_{\sigma},t)$ counts such sequences with the additional condition
that all elements are integers between $0$ and $t$. By defining
\[
  w_i = v_{\sigma^{-1}(i)} - \des( \sigma^{-1}(1),\dots, \sigma^{-1}(i) ),
\]
we have a bijection with integer sequences satisfying $0\leq w_1 \leq \dots \leq w_n \leq t-\des(\sigma^{-1})$,
so that
\[
  E(\mathcal{S}_{\sigma},t) = \binom { n + t - \des(\sigma^{-1}) }{n}.
\]
The expansion
\[
  \sum_{t\geq0 } \binom tn z^t = \frac{z^n}{(1-z)^{n+1} }
\]
permits to finish the proof.
\end{proof}

\begin{defi} The {\it Eulerian polynomials} are
\[
  A_n(z) := \sum_{\sigma \in \mathfrak{S}_n } z^{ \des(\sigma^{-1} ) } = \sum_{\sigma \in \mathfrak{S}_n } z^{ \des(\sigma ) }.
\]
Note that their coefficients are the numbers $A_{n,k}$ defined in Equation~\eqref{defank}.
\end{defi}

For example, $A_0(z)=A_1(z)=1$, $A_2(z)=1+z$, $A_3(z)=1+4z+z^2$, etc.
Since the Ehrhart series is additive with respect to disjoint union, we get from the previous lemma that
\[
  E^*( [0,1]^n , z ) =  A_n(z).
\]
Note that, since the Ehrhart polynomial of $[0,1]^n$ is clearly $(t+1)^n$, we have proved the classical 
identity:
\begin{equation}  \label{ideul}
  A_n(z) = (1-z)^{n+1} \sum_{t \geq 0 } (t+1)^n z^t.
\end{equation}

Let us turn to the case of the half-open hypercube $[0,1)^n$.
Note that we are now dealing with a half-open polytope, i.e., a polytope where 
some of the $(n-1)$-dimensional faces are removed. In this case, it is not {\it a priori} clear 
that the Ehrhart series is a polynomial with nonnegative integral coefficients.
We can decompose $[0,1)^n$ as the disjoint union of the polytopes:
\[
  \mathcal{T}_{\sigma} = \{ v \in [0,1)^n \, : \, \std(v)=\sigma \}.
\]
These are also simplices where some facets are removed, and we get:

\begin{lem} \label{ehrtsigma}
 $E^*(\mathcal{T}_\sigma,z) = z^{\des(\sigma^{-1}) + 1}$.
\end{lem}

\begin{proof}
It is similar to the one of Lemma~\ref{ehrssigma}.
\end{proof}

Thus, the half-open hypercube has the Ehrhart series $zA_n(z)$.
Besides, its Ehrhart polynomial is clearly $E([0,1)^n,t)=t^n$.
Once again we get Identity~\eqref{ideul}, with an additional factor $z$.

Let us present another example of a Ehrhart series that will be used in the sequel. It is 
presented in \cite[Section 7.19]{stanley2} in the context of quasi-symmetric functions.
Let $\lambda$ be a Young diagram (we use the French notation).
Let $\mathbb{Z}^{\lambda}$ (respectively, $\mathbb{R}^\lambda$) denote the set of fillings of $\lambda$ with 
integers (respectively, real numbers).
And let $\mathcal{Y}_{\lambda}$ denote the set of semi-standard fillings of $\lambda$ by real numbers in $(0,1]$
where semi-standard mean weakly increasing in rows and strictly increasing in columns.
Clearly, $\mathcal{Y}_{\lambda}$ is a (partially open) convex polytope in $\mathbb{R}^\lambda$.
A semi-standard tableau with largest entry less than $t$ is just an element
of $\mathbb{Z}^{\lambda} \cap t \mathcal{Y}_{\lambda} $, so that 
\[
  E(\mathcal{Y}_\lambda,t) = s_{\lambda}(1^t),
\]
the Schur function $s_{\lambda}$ where $t$
variables are set to $1$ and the others to $0$.
Let $SYT(\lambda)$ denote the set of standard tableaux of shape $\lambda$, and recall that 
a descent of a standard tableau is an entry $i$ such that the entry $i+1$ is in an upper row.
Let $\des(T)$ denote the number of descents of a standard tableau, then we have:

\begin{prop}
\begin{align} \label{sytdes}
  E^*( \mathcal{Y}_{\lambda} , z ) = (1-z)^{n+1} \sum_{ t\geq 0 } s_{\lambda}(1^t) z^t = \sum_{ T\in SYT(\lambda) } z^{\des(T)+1}.
\end{align}
\end{prop}

Let us sketch the proof.
The reading word $w(T)$ of a semi-standard tableau $T\in\mathcal{Y}_\lambda$ is defined by ordering 
its entries row by row, from left to right and from top to bottom.
Then, the standardization $\std(T)$ is defined to be the unique standard tableau $U$ 
of the same shape such that $\std(w(T)) = w(U) $.
The set $\mathcal{Y}_{\lambda}$ is partitioned into the subsets 
\[
  \mathcal{Y}_{U} = \{ T\in\mathcal{Y}_\lambda \, : \, \std(T)=U \}
\]
where $U\in SYT(\lambda)$. Now, the previous proposition is a consequence of the following:

\begin{lem}
 $ E^*( \mathcal{Y}_U , z ) = z^{\des(U)+1}$.
\end{lem}

\begin{proof}
This is essentially the same as Proposition~\ref{ehrtsigma}.
\end{proof}

\section{The generalization of Stanley's bijection}

\label{sec:stanbij}

In this section we adapt Stanley's bijection from \cite{stanley} to the case 
of the half-open hypercube $[0,r)^n$ (for some integer $r>0$), and $r$-colored permutations.
Note that a similar generalization was given by Steingr\'imsson in \cite[Section~4.4]{steingrimsson}.

\begin{defi}
 The set of $r$-{\it colored permutations} $\mathfrak{S}_n^{(r)}$ is the set of pairs
 $(\sigma,c)$ where $\sigma\in\mathfrak{S}_n$, $c=(c_i)_{1\leq i \leq n}$, and $c_i \in \{ 0,1, \dots,r-1 \}$ for all $i$.
 We define the {\it descent number} of a colored permutation as:
\[
  {\rm des}(\sigma,c) := \# \{ i \, : \, c_i>c_{i+1}, \text{ or } c_i=c_{i+1} \text{ and } \sigma_i>\sigma_{i+1} \},
\]
and we define the {\it flag descent number} \cite{bagno1,foatahan} as:
\[
  \fdes(\sigma,c) :=  r. {\rm des}(\sigma,c) + c_n.
\]
The {\it flag Eulerian numbers} are defined by $A^{(r)}_{0,0}:=1$, $A^{(r)}_{n,0}:=0$ if $n>0$, and
\[
  A_{n,k}^{(r)} := \# \big\{ \sigma \in \mathfrak{S}_n^{(r)} \, : \,  {\rm fdes}(\sigma) = k-1 \big\}
\]
if $n\geq 1$ and $1\leq k \leq rn$.
\end{defi}

We will not use here the group structure of colored permutations. Still, note that flag descents are indeed
related with it \cite{bagno1}.

To define our generalization of Stanley's bijection,
let $ (a_i)_{1\leq i \leq n} \in [0,r)^n$, and let $a_0=0$. 
Then the map $\phi( (a_i)_{1\leq i \leq n} ) = (b_i)_{1\leq i \leq n} $ is defined as follows:
\begin{equation} \label{defphi}
  b_i = \begin{cases}
          a_i-a_{i-1} & \text{ if } a_{i-1}\leq a_i, \\
          a_i-a_{i-1}+r & \text{ if } a_{i-1}> a_i.
        \end{cases}
\end{equation}
From this definition we get:
\begin{equation} \label{inversephi}
  a_i = \sum_{j=1}^i b_j  \mod r,
\end{equation}
where the modulo means that we take the unique representative in $[0,r)$.
In fact, it is elementary to check that Equations~\eqref{defphi} and~\eqref{inversephi} define 
two inverse bijections from $[0,r)^n$ to itself. 

\begin{defi}
If $v\in [0,r)^n$, we define $\fdes(v)=r.\des(v) +  v_n $. 
If $1 \leq k \leq rn $, let
\[
  \mathcal{F}^{(r)}_{n,k} := \{ v \in [0,r)^n \, : \, k-1 \leq \fdes(v) < k \},
\]
and
\[
  \mathcal{A}^{(r)}_{n,k} := \{ v \in [0,r)^n \, : \, k-1 \leq \sum v_i < k \}.
\]
For each colored permutation $(\sigma,c)$ we define the translated simplex:
\[
  \mathcal{T}_{(\sigma,c)} := c + \mathcal{T}_{\sigma}.
\]
Also, the {\it colored standardization} of $v\in [0,r)^n$ is $\cstd(v) = (\sigma,c) \in \mathfrak{S}^{(r)}_n$
where $c_i = \lfloor v_i \rfloor $ and $\sigma = \std( v_1 \mod 1,\dots, v_n \mod 1 )$.
\end{defi}

\begin{lem} \label{lem1}
$ \lfloor\fdes(v)\rfloor = \fdes( \cstd(v) )$. 
\end{lem}

\begin{proof}
This follows straighforwardly from the definitions.
\end{proof}

\begin{lem} \label{lem1b}
$\mathcal{A}^{(r)}_{n,k} = \phi(\mathcal{F}^{(r)}_{n,k})$.
\end{lem}

\begin{proof}
From \eqref{defphi} and keeping the notation we get
\[
  \sum_{i=1}^n b_i = r.\des(a_1,\dots,a_n) + a_n = \fdes(a_1,\dots,a_n)
\]
and the result follows.
\end{proof}

\begin{lem} \label{lem2}
$E^*( \phi( \mathcal{T}_{(\sigma,c)} ) , z ) = z^{\des(\sigma^{-1})+1}$.
\end{lem}

\begin{proof}
Let $v\in\mathcal{T}_{(\sigma,c)}$, then the condition $v_i\leq v_{i+1}$ or $v_i > v_{i+1}$ only depends on $(\sigma,c)$.
So from its definition in \eqref{defphi}, we see that the restriction of $\phi$ to $\mathcal{T}_{(\sigma,c)}$ is equal 
to an affine map that sends $\mathbb{Z}^n$ to itself. It follows that 
$\phi(\mathcal{T}_{(\sigma,c)})$ has the same Ehrhart series as $\mathcal{T}_{(\sigma,c)}$.
Besides, since $\mathcal{T}_{(\sigma,c)}$ is a translation of $\mathcal{T}_{\sigma}$ by an integer vector,
they have the same Ehrhart series, which is therefore $z^{\des(\sigma^{-1})+1}$ by Lemma~\ref{ehrtsigma}.
\end{proof}

\begin{prop} \label{ehrsA}
\[
   E^*(\mathcal{A}^{(r)}_{n,k} , z ) 
   =  \sum_{\substack{(\sigma,c) \in \mathfrak{S}_n^{(r)} \\ \fdes(\sigma,c) = k-1  }}  z^{\des(\sigma^{-1})+1}.
\]
\end{prop}

\begin{proof}
Using Lemma~\ref{lem1}, we get:
\[
  \mathcal{F}^{(r)}_{n,k} = \biguplus_{\substack{ (\sigma,c) \in \mathfrak{S}^{(r)}_n \\ \fdes(\sigma,c)=k-1  }} \mathcal{T}_{(\sigma,c)}.
\]
From Lemma~\ref{lem1b} and the fact that $\phi$ is a bijection, we have:
\[
    \mathcal{A}^{(r)}_{n,k} = \phi(\mathcal{F}^{(r)}_{n,k}) 
  = \biguplus_{\substack{ (\sigma,c) \in \mathfrak{S}^{(r)}_n \\ \fdes(\sigma,c)=k-1  }}  \phi( \mathcal{T}_{(\sigma,c)} ).
\]
From Lemma~\ref{lem2} and the fact that the Ehrhart series is additive with respect to disjoint union, we get the result.
\end{proof}

\section{From the half-open hypercube to the closed hypercube}

\label{sec:inclexcl}

\begin{defi}
The {\it multi-hypersimplices} are the polytopes defined by:
\[
  \mathcal{B}^{(r)}_{n,k} :=  
     \begin{cases} 
 \big\{ v\in [0,r]^n \, : \, k-1 \leq \sum v_i < k       \big\}  & \text{ if } 1 \leq k < rn, \\[3mm]
 \big\{ v\in [0,r]^n \, : \, k-1 \leq \sum v_i \leq k    \big\}  & \text{ if } k=rn.
     \end{cases}
\]
\end{defi}

These polytopes form a particular class of the ones introduced by Lam and Postnikov~\cite{lam} under the same name.
The polytopes $\mathcal{B}^{(1)}_{n,k}$ are simply called the hypersimplices, and they can be described
geometrically as truncated simplices (this fact is essentially due to Coxeter \cite[Section 8.7]{coxeter}).

Note that 
\[
  [0,r]^n = \biguplus_{ 1 \leq k \leq rn } \mathcal{B}^{(r)}_{n,k}.
\]


\begin{prop} \label{relab}
Let $B^{(r)}_{n,k}(z) = E^*( \mathcal{B}^{(r)}_{n,k} , z ) $ and $A^{(r)}_{n,k}(z) = E^*( \mathcal{A}^{(r)}_{n,k} , z ) $,
with the convention that $A_{n,0}^{(r)}(z) = B_{n,0}^{(r)}(z) = \delta_{n0}$
and $A_{n,k}^{(r)}(z) = B_{n,k}^{(r)}(z) = 0$ if $k<0$ or $k>rn$. Then:
\[
  B^{(r)}_{n,k}(z) = \sum_{j=0}^n \binom nj  (1-z)^j A^{(r)}_{n-j,k-rj}(z).
\]
\end{prop}

\begin{proof} 
For each $\Delta \subset \{1, \dots, n\}$, let 
$ \mathcal{H}_\Delta =  \{ v\in \mathcal{B}^{(r)}_{n,k} \, : \, v_i=r \text{ iff } i \in \Delta \} $.
The sets $\mathcal{H}_\Delta$ form a partition of $\mathcal{B}^{(r)}_{n,k}$, so that
\[
  B^{(r)}_{n,k}(z) = \sum_{\Delta \subset \{1,\dots,n\}}  E^*( \mathcal{H}_\Delta , z ).
\]
By removing the coordinates equal to $r$, we see that
$\mathcal{H}_\Delta$ is in bijection with $\mathcal{A}^{(r)}_{n-j,k-rj}$ where $j=\#\Delta$. 
The bijection preserves integral points, and this holds with the convention that both polytopes
are empty if $k-rj<0$. Hence:
\[
  E^*(  \mathcal{H}_\Delta , z ) = (1-z)^j A^{(r)}_{n-j,k-rj}(z).
\]
\end{proof}


The previous proposition is conveniently rewritten in terms of the generating functions.
Let
\[
  A^{(r)}(x,y,z) = \sum_{n\geq 0} \left( \sum_{k=0}^{rn} A^{(r)}_{n,k}(z) y^k \right) \frac{x^n}{n!},
  \quad
  B^{(r)}(x,y,z) = \sum_{n\geq 0} \left( \sum_{k=0}^{rn} B^{(r)}_{n,k}(z) y^k \right) \frac{x^n}{n!},
\]
then these two series are related as stated below.

\begin{theo} \label{relAB}
The following identity holds:
\[ B^{(r)}(x,y,z) = e^{(1-z)y^rx} A^{(r)}(x,y,z). \]
\end{theo}

\begin{proof}
From Proposition~\ref{relab}, we get:
\[
  \sum_{k \geq 0} y^k B^{(r)}_{n,k}(z) = \sum_{j=0}^n \binom nj ((1-z)y^r)^j \sum_{k \geq 0}  y^{k-rj} A^{(r)}_{n-j,k-rj}(z)
\]
and the result follows.
\end{proof}

Together with Proposition~\ref{ehrsA}, the relation in the previous theorem shows that 
a generalization of Stanley's conjecture can be obtained via an identity on generating functions. 
This identity will be presented in the next sections. 
We first need some definitions to state the result.

\begin{defi} [\cite{bagno2,foatahan}]
The {\it flag excedence number} of a colored permutation is:
\[
  {\rm fexc}(\sigma,c) := r.\#\big\{ i \in\{1,\dots,n\} \, : \, \sigma_i>i\text{ and } c_i=0 \big\} + \sum_{i=1}^{n} c_i.
\]
We also need another definition of flag descents, which is the one originally due to Foata and Han \cite{foatahan}:
\[
  \fdes^*(\sigma,c) := r . \des^* (\sigma,c) + c_1
\]
where 
\[
  \des^*(\sigma,c) := \# \{ i \, : \, c_i < c_{i+1}, \text{ or } c_i=c_{i+1} \text{ and } \sigma_i>\sigma_{i+1} \}.
\] 
\end{defi}

In particular, let us mention that the statistics $\fexc$ and $\fdes^*$ are equidistributed on $\mathfrak{S}_n^{(r)}$,
see \cite[Theorem~1.4]{foatahan}. We will also give a proof that $\fdes$ and $\fdes^*$ are equidistributed in the next
section. 

\begin{theo} \label{theobnr}
\begin{equation} \label{eq:theobnr}
   B^{(r)}_{n,k}(z) = 
   \sum_{ \substack{ (\sigma,c) \in \mathfrak{S}^{(r)}_n \\ \fexc(\sigma,c) = rn-k   }}  z^{ \lceil \fdes^*(\sigma,c) /r \rceil }.
\end{equation}
\end{theo}

\begin{proof}
Let $C_{n,k}^{(r)}(z)$ denote the right-hand side of the equation, and we use the same convention
as with $B_{n,k}^{(r)}(z)$ when $k\leq0$. Let also $C_n^{(r)}(y,z) = \sum_{k=0}^{rn} C^{(r)}_{n,k}(z) y^k $, and
\[
  C^{(r)}(x,y,z) = \sum_{n\geq 0} C_n^{(r)}(y,z) \frac{x^n}{n!}.
\]
From Proposition~\ref{ehrsA} and Theorem~\ref{relAC} in the sequel, we have
\[
  C^{(r)}(x,y,z) = e^{(1-z)y^rx} A^{(r)}(x,y,z). 
\]
Comparing with Theorem~\ref{relAB} shows that we have 
$B^{(r)}(x,y,z) = C^{(r)}(x,y,z)$, which proves the theorem.
\end{proof}

We have in fact another result, which is not trivially equivalent to the previous one:

\begin{theo} \label{theobnr2}
\[
   B^{(r)}_{n,k}(z) = 
   \sum_{ \substack{ (\sigma,c) \in \mathfrak{S}^{(r)}_n \\ \fexc(\sigma,c) = rn-k   }}  z^{ \lceil \fdes(\sigma,c) /r \rceil }.
\]
\end{theo}

\begin{proof}
This is a consequence of the previous theorem, together with the bijection in Section~\ref{sec:bij}.
\end{proof}

In view of the previous two theorems, one can ask whether the pairs $(\fexc,\fdes)$ and $(\fexc,\fdes^*)$
are equidistributed. This is however not the case.

\section{Chromatic descents}

\label{sec:chro}


\begin{defi}
For a colored permutation $(\sigma,c)$ we define its {\it chromatic descent number} as
\[
  \cdes(\sigma,c) := \des(\sigma) + \sum_{i=1}^n c_i.
\]
\end{defi}

We show that it is equidistributed with the flag descent number, via a bijection $\alpha$.
Let $(\sigma,c)$ be a colored permutation. Let $\alpha(\sigma,c) = (\sigma,c')$ where
\[
   c'_i = \sum_{j=1}^i  c_j + \des(\sigma_1,\dots,\sigma_i) \mod r.
\]

\begin{prop}
$\fdes(\sigma,c') = \cdes(\sigma,c)$.
\end{prop}

\begin{proof}
Let $w_k = \sum_{j=1}^k c_j + \des(\sigma_1,\dots,\sigma_k)$ for $k=1,\dots,n$ so that $c'_i = w_i \mod r$.
Clearly, $w_1,\dots,w_n$ is a nondecreasing sequence and $w_n=\cdes(\sigma,c)$.
We can write $w_n=qr+c_n'$ for a unique $q$.
This integer $q$ counts the number of positive multiples of $r$ that are smaller than $w_n$.
Using the fact that $w_k-w_{k-1}\leq r$ and $w_1<r$, we have:
\[
   q = \#\{ i \, : \,  \exists k,   w_{i-1} < kr \leq w_i   \}.
\]
To count the cardinality of this set, we distinguish two cases.
If $w_{i-1}-w_i=r$, it means that $c_i=r-1$ and $\sigma_{i-1}>\sigma_i$. 
From the definition of the bijection, this is equivalent to $c'_i=c'_{i+1}$ and $\sigma_{i-1}>\sigma_i$.
Otherwise, $w_{i-1}-w_i<r$. We can see that this case is equivalent to $c_{i-1}>c_{i}$.
Hence, we obtain $q=\des(\sigma,c')$, and $w_n=\fdes(\sigma,c')$.
\end{proof}

It is also possible to define a bijection $\alpha^*$ by $\alpha^*(\sigma,c)=(\sigma,c'')$ where
\[
  c''_i = \sum_{j=i}^n c_j + \des(\sigma_i,\dots,\sigma_n) \mod r.
\]
As in the case of the previous proposition, we can prove $\fdes^*(\sigma,c'') = \cdes(\sigma,c)$.
In particular, it follows that $\fdes$ and $\fdes^*$ are equidistributed.

The bijection $\alpha$ only changes the colors $c_i$, and not the permutation $\sigma$, so we have:
\[
    \sum_{(\sigma,c) \in \mathfrak{S}_n^{(r)} } y^{ \fdes(\sigma,c) } z^{ \des(\sigma^{-1}) } 
  = \sum_{(\sigma,c) \in \mathfrak{S}_n^{(r)} } y^{ \cdes(\sigma,c) } z^{ \des(\sigma^{-1}) } .
\]
But the right-hand side clearly can be factorized, so that with the notation
\[
  A^{(r)}_n(y,z) = \sum_{ k=0 }^{ rn } y^k A^{(r)}_{n,k}(z),
\]
we have:
\begin{equation} \label{relara1}
  A^{(r)}_n(y,z) = \left( \frac{1-y^r}{1-y} \right)^n A^{(1)}_n(y,z).
\end{equation}

A formula for the case $r=1$ is given in the proposition below.
This is in fact a  particular case of a result of Garsia and Gessel \cite[Theorem~2.3]{garsia},
but we also include a short proof based on the Robinson-Schensted correspondence.

\begin{prop} For $r=1$, we have:
\[
  \frac{A^{(1)}_n(y,z)}{(1-y)^{n+1}(1-z)^{n+1}} = \sum_{i,j \geq 0} \binom{ij+n-1}{n} y^i z^j.
\]
\end{prop}

\begin{proof}
Let $\text{Par}(n)$ denote the set of integer partitions of $n$.
By the Robinson-Schensted correspondence, we have:
\begin{align*}
  A^{(1)}_n(y,z) &= \sum_{\sigma \in \mathfrak{S}_n } y^{\des(\sigma)+1} z^{\des(\sigma^{-1})+1} \\
                 &= \sum_{ \lambda \in {\rm Par}(n) } \sum_{P,Q \in SYT(\lambda)} y^{\des(P)+1} z^{\des(Q)+1}.
\end{align*}

So, using Equation~\eqref{sytdes}, we get:
\[
    \frac{A^{(1)}(y,z)}{ (1-y)^{n+1}(1-z)^{n+1} } 
  =  \sum_{ \lambda \in {\rm Par}(n) } \sum_{ s,t \geq 0 } s_{\lambda}(1^s) s_{\lambda}(1^t) y^s z^t.
\]
By the Cauchy identity on Schur functions, we have 
\[
    \sum_{ \lambda \in {\rm Par}(n) } s_\lambda(1^s) s_\lambda(1^t) = [x^n] \Big( \frac{1}{1-x} \Big)^{st}
  = \binom{st+n-1}{n}.
\]
This ends the proof.
\end{proof}


From Equation~\eqref{relara1} and the previous proposition, we deduce:

\begin{prop} 
\begin{equation} \label{formuleanr}
  \frac{A^{(r)}_n(y,z)}{ (1-y^r)^n (1-y) (1-z)^{n+1} } = \sum_{i,j \geq 0} \binom{ij+n-1}{n} y^i z^j.
\end{equation}
\end{prop}

%
%
%
Note that another consequence of Equation~\eqref{relara1}, together with Equation~\eqref{ideul}, is
the following (which is not a new result, see for example \cite{bagno1}).
\begin{prop}
\begin{equation}  \label{flagpol}
  \frac{A^{(r)}_n(y,1)}{ (1-y^r)^n (1-y) } = \sum_{i \geq 1} i^n y^i.
\end{equation}
\end{prop}
%
%
%

\section{Identities on bi-Eulerian generating functions}

\label{sec:genfunc}

We keep the definition of $A_n^{(r)}(y,z)$ and $A^{(r)}(x,y,z)$ as before,
but in this section we only need the formula in Equation~\eqref{formuleanr}.
We recall that $C_{n,k}^{(r)}(y,z)$, $C_n^{(r)}(y,z)$, and $C(x,y,z)$ were defined in 
the proof of Theorem~\ref{theobnr}.
The goal of this section is to prove the following relation
between the two generating functions for colored permutations:

\begin{theo} \label{relAC}
$ C^{(r)}(x,y,z) = e^{(1-z)y^rx} A^{(r)}(x,y,z)$.
\end{theo}

Let us define 
\[
  W_n(y,z) := \sum_{(\sigma,c) \in \mathfrak{S}_n^{(r)}}  y^{ \fexc(\sigma,c)} z^{\fdes^*(\sigma,c)}.
\]
The particular case $q=1$ of \cite[Theorem 5.11]{foatahan}, after an easy simplification, gives the following formula:
\begin{equation} \label{formulawn}
  \sum_{n\geq 0  }  W_n(y,z) \frac{x^n}{(1-z^r)^n}  = (1-z)  \sum_{ k \geq 0 }  z^k F_k(x,y) , 
\end{equation}
where
\[
  F_k(x,y) = \frac{ (1-xy^r)^{\lfloor k/r\rfloor} } {(1-x)^{\lfloor k/r\rfloor+1}} (1-y^r)
 \times
\Bigl(  \frac{1-y^r}{1-y} - \sum_{i=1}^{r} y^i 
{  \frac{ (1-xy^r)^{\lfloor (k-i)/r\rfloor+1} } { (1-x)^{\lfloor (k-i)/r\rfloor+1}} } \Bigr)^{-1}. 
\]
Next, we define a linear operator $\beta$ on power series in $z$ by $\beta(z^k)= z^{\lceil k/r \rceil }$.
So:
\[
   \beta( z^k - z^{k+1} )  = \begin{cases}
                                0 & \text{ if } k \not\equiv 0 \mod r, \\
                                z^m-z^{m+1} & \text{ if } k=rm.
                             \end{cases}
\]
From Equation~\eqref{formulawn}, we get
\begin{equation} \label{formulabetawn}
  \sum_{n\geq 0  }  \beta( W_n(y,z) ) \frac{x^n}{(1-z)^n}  = \sum_{ m \geq 0 } ( z^m - z^{m+1} ) F_{rm}(x,y),
\end{equation}
and from the definition of $F_k$ we get
\begin{align*}
  F_{rm}(x,y) &= 
     \frac{(1-xy^r)^m}{ (1-x)^{m+1}} (1-y^r)
\Bigl(
 \frac{1-y^r}{1-y}
- \sum_{i=1}^{r}
y^i \frac {(1-xy^r)^{m} }
  { (1-x)^{m}} \Bigr)^{-1} 
\\
 & = \Bigl(  \frac{(1-x)^{m+1}}{(1-xy^r)^m(1-y)}  - \sum_{i=1}^r y^i \frac{1-x}{1-y^r}   \Bigl)^{-1}
\\
 & = \frac{1-y}{1-x}
 \Bigl(
 \bigl( \frac{1-x}{1-xy^r}\bigr)^m
-y\Bigr)^{-1}.
\end{align*}
From Equation~\eqref{formulabetawn} and the previous equation, and after the substitution 
$(x,y)  \leftarrow  (xy^r(1-z) , y^{-1} ) $, we reach:

\begin{theo} \label{ogfC}
\[
  \sum_{n\geq 0 } C_n^{(r)}(y,z) x^n = 
  \frac{(1-z)(1-y)}{1-xy^r(1-z)}
 \sum_{m\ge 0}z^m\,
 \Bigl(1-
 y\bigl( \frac {1-xy^r(1-z) }{ 1-x(1-z)} \bigr)^m
 \Bigr)^{-1}.
\]
\end{theo}

Besides, from Equation~\eqref{formuleanr}, we have: 
\begin{align*}
    \sum_{n\geq0} x^n \frac{A_n^{(r)}(y,z)}{ (1-y^r)^n(1-y)(1-z)^{n+1}}
      &=\sum_{i,j \geq0 } \Bigl(\frac{1}{1-x}\Bigr)^{ij} y^i z^j
       =  \sum_{j\geq0} z^j \Bigl(1-y\bigl( \frac{1}{1-x}\bigr)^j   \Bigr)^{-1}.
\end{align*}
After the substition $x\leftarrow x(1-y^r)(1-z)$, we obtain:

\begin{theo} \label{ogfA}
\[
 \sum_{n\geq0} {A_n^{(r)}(y,z)} x^n =
 (1-y)(1-z)\sum_{m\geq 0} z^m \Bigl(1-y\bigl( \frac{1}{1-x(1-y^r)(1-z)}\bigr)^m   \Bigr)^{-1}.
\]
\end{theo}

\begin{proof}[Proof of Theorem \ref{relAC}]
Since Theorem~\ref{relAC} is a relation on exponential generating functions, it is convenient to use
the Laplace transform. It sends a function $f(x)$ to
\[
  \mathcal{L}(f(x),x,s) = \int_0^\infty f(x) e^{-xs} {\rm d}x,
\]
in particular,
\[
  \mathcal{L} \Big( \frac{x^k}{k!} , x , s \Big) = \frac{1}{s^{k+1}}.
\]
We have:
\begin{align*}
  \mathcal{L} \big( e^{(1-z)y^rx} A^{(r)}(x,y,z) , x, s \big)  & = \int_0^{\infty} e^{(1-z)y^rx} A^{(r)}(x,y,z) e^{-xs} {\rm d}x \\
                                                               & = \int_0^{\infty}  A^{(r)}(x,y,z) e^{(1-z)y^rx-xs} {\rm d}x \\
                                                               & = \mathcal{L}\big( A^{(r)}(x,y,z) , x , s-(1-z)y^r \big).
\end{align*}
By Theorem~\ref{ogfA}, with $s'=s-(1-z)y^r$, the latter expression is equal to
\[
  \frac{(1-y)(1-z)}{s'}
  \sum_{m\geq0} z^m \Bigl(1-y\bigl(\frac{1}{1-\frac{1}{s'}(1-y^r)(1-z)}\bigr)^m   \Bigr)^{-1}.
\]
Since 
\[
\frac{1}{1-\frac{1}{s'}(1-y^r)(1-z)}
=
\frac{1}{1-\frac{1}{s-(1-z)y^r}(1-y^r)(1-z)}
=
\frac{s-(1-z)y^r}{s-(1-z)},
\]
we get:
\[
  \mathcal{L} \big( e^{(1-z)y^rx} A^{(r)}(x,y,z) , x, s \big) = 
  \frac{(1-y)(1-z)}{s-(1-z)y^r}\sum_{m\geq 0} z^m \Bigl(1-y\bigl(\frac{s-(1-z)y^r}{s-(1-z)}\bigr)^m   \Bigr)^{-1}.
\]
Besides, from Theorem~\ref{ogfC}, we also get:
\[
  \mathcal{L} \big( C^{(r)}(x,y,z) , x, s \big) = 
  \frac{(1-y)(1-z)}{s-(1-z)y^r}\sum_{m\geq 0} z^m \Bigl(1-y\bigl(\frac{s-(1-z)y^r}{s-(1-z)}\bigr)^m   \Bigr)^{-1}  .
\]
So we have proved 
\[
  \mathcal{L} \big( C^{(r)}(x,y,z) , x, s \big) = \mathcal{L}\big( e^{(1-z)y^rx} A^{(r)}(x,y,z) , x , s \big),
\]
which completes the proof of Theorem~\ref{relAC}.
\end{proof}

\section{Another combinatorial model}

\label{sec:bij}

We give in this section a bijective proof of
\[
    \sum_{ (\sigma,c) \in \mathfrak{S}_n^{(r)} } y^{ {\fexc}(\sigma,c)  } z^{ \lceil {\fdes}(\sigma,c) /r \rceil }
  = \sum_{ (\sigma,c) \in \mathfrak{S}_n^{(r)} } y^{ {\fexc}(\sigma,c)  } z^{ \lceil {\fdes^*}(\sigma,c) /r \rceil },
\]
by defining an involution $I$ on colored permutations such that
\[
  y^{ {\fexc}(\sigma,c)  } z^{ \lceil {\fdes}(\sigma,c) /r \rceil } = y^{ {\fexc}(I(\sigma,c))  } z^{ \lceil {\fdes^*}(I(\sigma,c)) /r \rceil }.  
\]
Let $(\sigma,c) \in \mathfrak{S}^{(r)}_n$. We consider $(\sigma,c)$ as a word whose successive
letters are $(\sigma_1,c_1)$, $(\sigma_2,c_2),\dots$, $(\sigma_n,c_n)$. 
Note that the pair $(\sigma_i,c_i)$ is considered as a letter with color $c_i$.
Then, we consider the unique factorization 
\[
  (\sigma_1,c_1) \dots (\sigma_n,c_n) =  B_1 \dots B_m
\]
where each block $B_i$ contains letters of the same color, and $m$ is minimal.
The involution $I$ is defined by permuting the blocks, following these two conditions:
\begin{itemize}
 \item each zero colored block stays at the same location,
 \item each maximal sequence of nonzero colored blocks $B_j \dots B_k$ is replaced with $B_k \dots B_j$
       (maximal means that $B_{j-1}$ is zero colored or $j=1$, and $B_{k+1}$ is zero colored or $k=n$).
\end{itemize}
For example, with $n=8$ and $r=3$:
\[
  (8,1) (2,0) (7,2) (1,2) (4,1) (3,0) (5,1) (6,1)
\]
is sent to 
\[
  (8,1) (2,0) (4,1) (7,2) (1,2) (3,0) (5,1) (6,1).
\]

\begin{lem}
 $\fexc(\sigma,c) = \fexc(I(\sigma,c))$.
\end{lem}

\begin{proof}
This is immediate, since the letters with color $0$ are unchanged by $I$
(and the sum of the colors is also preserved).
\end{proof}

\begin{lem}
 $\lceil \fdes(\sigma,c) /r \rceil = \lceil \fdes^*(I(\sigma,c)) /r \rceil$.
\end{lem}

\begin{proof}
We compute $\lceil \fdes(\sigma,c) /r \rceil$ on one side and $\lceil \fdes^*(I(\sigma,c)) /r \rceil$ and the other side,
by examining the different contributions to each quantity.

First, each pair of letters $(\sigma_i,c_i) (\sigma_{i+1},c_{i+1})$ where $\sigma_i>\sigma_{i+1}$ inside a given block $B_j$ 
contribute by 1 to each side (since $c_i=c_{i+1}$ by definition of the blocks). It remains to consider the term
$r\times\#\{ i \, :\, c_i > c_{i+1}  \} + c_n$ in the definition of $\fdes$, and the term
$r\times\#\{ i \, :\, c_i < c_{i+1}  \} + c_1$ in the definition of $\fdes^*$.

Let us write $B_i>B_{i+1}$ or $B_i<B_{i+1}$ to mean that the color of the block $B_i$ is greater or smaller than that of 
$B_{i+1}$ (by definition they cannot be equal).
Let $j<k$ be such that $B_j$ and $B_k$ are zero colored blocks, but $B_{j+1},\dots,B_{k-1}$ are not.
In the factor $B_j\dots B_k$ of $(\sigma,c)$, there is a contribution 
\[
  \#\big\{ i \,:\,  j\leq i <k \text{ and } B_i > B_{i+1}  \big\}
\]
to $\lceil \fdes(\sigma,c) /r \rceil$. But in the factor $B_j B_{k-1}\dots B_{j+1} B_k$ of $I(\sigma,c)$,
there is the same contribution to $\lceil \fdes^*(I(\sigma,c)) /r \rceil$.

Now, let $B_j$ be the first zero colored block of $(\sigma,c)$. If $j>1$, the prefix $B_1\dots B_j$ of $(\sigma,c)$ 
contributes by 
\[
  \#\big\{ i \,:\,  1\leq i <j \text{ and } B_i > B_{i+1}  \big\}
\]
to $\lceil \fdes(\sigma,c) /r \rceil$, and the prefix $B_{j-1}\dots B_1 B_j$ of $I(\sigma,c)$ contributes by
\[
  1 + \#\big\{ i \,:\,  j-1 > i \geq 1 \text{ and } B_{i+1} < B_i  \big\}
\]
to $\lceil \fdes^*(I(\sigma,c)) /r \rceil$ (the 1 come from the term $c_1$ in the definition of $\fdes^*$
since $B_{j-1}$ is a nonzero colored block). The two numbers are easily seen to be equal.

Similarly, let $B_k$ be the last zero colored block of $(\sigma,c)$. If $k<m$, the suffix $B_k\dots B_m$ of $(\sigma,c)$
contributes by
\[
  1+ \#\big\{ i \,:\,  k \leq i <m \text{ and } B_i > B_{i+1}  \big\}
\]
to $\lceil \fdes(\sigma,c) /r \rceil$ (the 1 come from the term $c_n$ in the definition of $\fdes$
since $B_m$ is a nonzero colored block), and the suffix $B_k B_ m \dots B_{k+1}$ of $I(\sigma,c)$ contributes by
\[
  1 + \#\big\{ i \,:\,  m > i \geq k+1 \text{ and } B_{i+1} < B_i  \big\}
\]
to $\lceil \fdes^*(I(\sigma,c)) /r \rceil$ (the 1 come from the fact that $B_k<B_m$). 
The two numbers are easily seen to be equal.

Checking the respective definitions of $\fdes$ and $\fdes^*$, we can see that what we have counted proves the proposition.
\end{proof}

\section{Formulas for the Ehrhart polynomials}

\label{sec:ehrform}

\begin{theo}
The Ehrhart polynomial of $\mathcal{A}^{(r)}_{n,k}$ is:
\begin{align*}
   \sum_{j=0}^{\lfloor (k-1)/r \rfloor} (-1)^{j+1}  \tbinom nj \tbinom{n-rtj+kt-t-1}{n} -  
                      \sum_{j=0}^{\lfloor k/r \rfloor  } (-1)^{j+1}  \tbinom nj \tbinom{n-rtj+kt-1}{n}.
\end{align*}
\end{theo}

\begin{proof}
Let ${\rm CT}_q$ denote the operator that gives the constant term of a Laurent series in $q$. We have:
\[
   \# \big( \mathbb{Z}^n \cap t\mathcal{A}^{(r)}_{n,k} \big)  
                   = \# \big\{ v \in \{0,1,\dots,rt-1\}^n  \, : \,  kt-t \leq \sum v_i < kt \big\}
\]

\[
   =  {\rm CT}_q \bigg(  [rt]_q^n  ( [kt]_{q^{-1}} - [kt-t]_{q^{-1}} )  \bigg)
\]

\[
   = {\rm CT}_q \bigg(  \frac{ (1-q^{rt})^n  ( q^{-kt+t-1} - q^{-kt-1 } )   }{ (1-q)^n (1-q^{-1}) }   \bigg)
\]

\[
   = {\rm CT}_q \bigg(  \sum_{j=0}^n \binom nj (-1)^{j+1}  q^{rtj+j}  \frac {  ( q^{-kt+t} - q^{-kt } )   }{ (1-q)^{n+1} }   \bigg)
\]

\[
   = {\rm CT}_q \bigg(  \sum_{j=0}^n \sum_{i\geq0}  \binom nj \binom{n+i}{n}  (-1)^{j+1}  q^{rtj+j+i}   ( q^{-kt+t} - q^{-kt } )  \bigg)
\]

\[
   = \sum_{j=0}^n  \binom nj   (-1)^{j+1}  \left( \binom{n+kt+t-(rt+1)j}{n} - \binom{n+kt-(rt+1)j}{n}  \right)
\]
with the (unusual) convention that $\binom nk=0$ when $n<0$.
With this convention, it is not clear that we have a polynomial in $t$.
But we can improve the formula by keeping only some of the indices $j$, those appearing
in the announced formula.
Indeed both formulas are equal for large $t$, hence for every $t$ since these are polynomials.
\end{proof}

\begin{theo}
The Ehrhart polynomial of $\mathcal{B}^{(r)}_{n,k}$ is:
\begin{align*}
   \sum_{j=0}^{\lfloor (k-1)/r \rfloor} (-1)^{j+1}  \tbinom nj \tbinom{n-rtj-j+kt-t-1}{n} - 
                      \sum_{j=0}^{\lfloor k/r \rfloor  } (-1)^{j+1}  \tbinom nj \tbinom{n-rtj-j+kt-1}{n}.
\end{align*}
\end{theo}

\begin{proof}
This is similar to the previous proposition:
\[
   \# \big( \mathbb{Z}^n \cap t\mathcal{B}^{(r)}_{n,k} \big)  
                   = \# \big\{ v \in \{0,1,\dots,rt\}^n  \, : \,  kt-t \leq \sum v_i < kt \big\}
\]

\[
   =  {\rm CT}_q \bigg(  [rt+1]_q^n  ( [kt]_{q^{-1}} - [kt-t]_{q^{-1}} )  \bigg)
\]

\[
   = {\rm CT}_q \bigg(  \frac{ (1-q^{rt+1})^n  ( q^{-kt+t} - q^{-kt } )   }{ (1-q)^n (1-q^{-1}) }   \bigg)
\]

\[
   = {\rm CT}_q \bigg(  \sum_{j=0}^n \binom nj (-1)^{j+1}  q^{rtj+j}  \frac {  ( q^{-kt+t+1} - q^{-kt+1} )   }{ (1-q)^{n+1} }   \bigg)
\]

\[
   = {\rm CT}_q \bigg(  \sum_{j=0}^n \sum_{i\geq0}  \binom nj \binom{n+i}{n}  (-1)^{j+1}  q^{rtj+j+i}   ( q^{-kt+t+1} - q^{-kt+1 } )  \bigg)
\]

\[
   = \sum_{j=0}^n  \binom nj   (-1)^{j+1}  \left( \binom{n-rtj-j+kt-t-1}{n} - \binom{n-rtj-j+kt-1}{n}  \right).
\]
As in the previous case, the formula is obtained with the convention that $\binom nk=0$ when $n<0$,
but is true in general.
\end{proof}

We also obtain a formula for the flag Eulerian numbers.

\begin{theo}
The flag Eulerian number is: 
\[
   A^{(r)}_{n,k} = 
        \sum_{j=0}^{ \lfloor (k-1)/r \rfloor } \binom nj  (-1)^{j+1}  (k-rj-1)^n  
      - \sum_{j=0}^{ \lfloor k/r \rfloor } \binom nj      (-1)^{j+1}  (k-rj)^n.
\] 
\end{theo}

\begin{proof}
The number $ \text{Vol}(\mathcal{A}^{(r)}_{n,k}) = \frac{1}{n!} A^{(r)}_{n,k}$ can be obtained as the dominant 
coefficient of the Ehrhart polynomial.
From the exact formula we have just obtained, this dominant coefficient is:
\[
     \sum_{j=0}^{ \lfloor (k-1)/r \rfloor } \binom nj  \frac{ (-1)^{j+1} }{n!}  (k-rj-1)^n  
   - \sum_{j=0}^{ \lfloor     k/r \rfloor } \binom nj  \frac{ (-1)^{j+1} }{n!}  (k-rj)^n.
\]
This is the announced formula up to the normalization factor $n!$.
This could also be obtained from \eqref{flagpol}.
\end{proof}

Note that the particular case $r=1$ gives a well-known formula:
\begin{align*} 
  A_{n,k} &= \sum_{j=0}^k \binom nj (-1)^{j+1} ( (k-j)^n - (k-j+1)^n )  \\
          &= \sum_{j=1}^{k+1} \binom n{j-1} (-1)^{j} (k-j+1)^n - \sum_{j=0}^{k} \binom nj (-1)^{j+1} (k-j+1)^n \\
          &= \sum_{j=0}^k  \binom{n+1}{j} (-1)^j (k-j+1)^n.
\end{align*}

\section{Bijective problems}

In this article we have obtained a combinatorial interpretation of $E^*(\mathcal{B}_{n,k}^{(r)},z)$
which differs from the one previously obtained by Li \cite{li}. 
It would be interesting to have a bijective proof that the two results are equivalent.
For convenience, let us state Li's result here. We make the convention that $\sigma(0)=\sigma^{-1}(0)=0$ for
each permutation $\sigma\in\mathfrak{S}_n$.

\begin{defi}[Li \cite{li}]
 The statistic ${\rm cover}(\sigma)$ of a permutation $\sigma$ is defined by
\[
  {\rm cover}(\sigma) := \# \big\{ i \, : \, 1 \leq i \leq n \text{ and } \sigma^{-1}(i-1) + 1 < \sigma^{-1}(i) \big\},
\]
and the statistic ${\rm cef}(\sigma,c)$ of a colored permutation $(\sigma,c)$ is defined by
\[
  {\rm cef}(\sigma,c) := \# \big\{ i  \,:\,  1\leq i \leq n, \, c_i>0,  \text{ and } \sigma(i-1)+1=\sigma(i)  \big\}.
\]
\end{defi}

Although we have used slightly different conventions, it is easily seen that Theorem~7.3 from \cite{li} can
be stated as follows.

\begin{theo}[Li \cite{li}] \label{thli}
We have:
\[
  E^*(\mathcal{B}_{n,k}^{(r)},z) = \sum_{ \substack{ (\sigma,c) \in \mathfrak{S}_n^{(r)}  \\ \cdes(\sigma,c)=rn-k     } } 
                     z^{ {\rm cover}(\sigma) + {\rm cef}(\sigma,c)   }.
\]
\end{theo}

From Theorem~\ref{theobnr2} and Theorem~\ref{thli}, we obtain the equality
\[
  \sum_{ \substack{ (\sigma,c) \in \mathfrak{S}^{(r)}_n \\ \fexc(\sigma,c)= k }} z^{ \lceil \fdes(\sigma,c) /r \rceil }
  =
  \sum_{ \substack{ (\sigma,c) \in \mathfrak{S}_n^{(r)} \\ \cdes(\sigma,c)= k }}  z^{ {\rm cover}(\sigma) + {\rm cef}(\sigma,c) },
\]
which appeals for a bijective proof.

The second problem is to find a bijective proof of Theorem~\ref{relAC}, i.e., of the relation:
\begin{align*}
  e^{zy^rx} \left( 1 + \sum_{n\geq1} \sum_{(\sigma,c) \in \mathfrak{S}_n^{(r)}}  
                   y^{rn-\fexc(\sigma,c)} z^{\lceil\fdes(\sigma,c)/r\rceil} \frac{x^n}{n!} \right) \\
  = 
  e^{y^rx} \left( 1 + \sum_{n\geq 1} \sum_{(\sigma,c) \in \mathfrak{S}_n^{(r)}}  
                  y^{\fdes(\sigma,c)+1} z^{\des(\sigma^{-1})+1} \frac{x^n}{n!}\right).
\end{align*}

%
%
%
%


\bigskip

\setlength{\parindent}{0pt}

\end{document}